\documentclass{article}
\usepackage[utf8]{inputenc}
\usepackage{amssymb,amstext,amsmath,amscd,amsthm,amsfonts,enumerate,latexsym}
\usepackage{color}
\usepackage{stmaryrd}

\usepackage{titlesec}
\titleformat{\section}{\centering\large\sc}{\thesection.}{0.7em}{}
\usepackage{etoolbox}
\patchcmd{\abstract}
  {\bfseries\abstractname}
  {\scshape\abstractname}
  {}{}

\usepackage{hyperref}
\hypersetup{
    colorlinks=true,
    linkcolor=blue,
    citecolor=blue,
    filecolor=blue,      
    urlcolor=blue,
} 
\usepackage{bookmark}
\usepackage[backend=biber,style=alphabetic,sorting=nyt,maxbibnames=99,maxalphanames=99]{biblatex}

\newcommand{\sortas}[1]{}
    \usepackage{afterpage}
\usepackage[all]{xy}
\usepackage{tikz-cd}

\begin{document}
\theoremstyle{plain}
\newtheorem{thm}{Theorem}[section]
\newtheorem{theorem}[thm]{Theorem}

\numberwithin{equation}{thm}
\newtheorem*{thm*}{Theorem}
\newtheorem*{cor*}{Corollary}
\newtheorem*{thma}{Theorem A}
\newtheorem*{thmb}{Theorem B}
\newtheorem*{mthm}{Main Theorem}
\newtheorem*{mcor}{Theorem \ref{maincor}}
\newtheorem{thms}[thm]{Theorems}
\newtheorem{prop}[thm]{Proposition}
\newtheorem{proposition}[thm]{Proposition}
\newtheorem{prodef}[thm]{Proposition and Definition}
\newtheorem{lemma}[thm]{Lemma}
\newtheorem{lem}[thm]{Lemma}
\newtheorem{cor}[thm]{Corollary}

\newtheorem{corollary}[thm]{Corollary}
\newtheorem{claim}[thm]{Claim}
\newtheorem*{claim*}{Claim}

\theoremstyle{definition}
\newtheorem{defn}[thm]{Definition}
\newtheorem{defnProp}[thm]{Definition and Proposition}
\newtheorem{prob}[thm]{Problem}
\newtheorem{question}[thm]{Question}
\newtheorem{definition}[thm]{Definition}
\newtheorem{ex}[thm]{Example}
\newtheorem{exercise}[thm]{Exercise}
\newtheorem{Example}[thm]{Example}
\newtheorem*{notation}{Notation}
\newtheorem*{acknowledgements}{Acknowledgements}
\newtheorem{example}[thm]{Example}
\newtheorem{fact}[thm]{Fact}
\newtheorem*{Fact*}{Fact}
\newtheorem{conj}[thm]{Conjecture}
\newtheorem{ques}[thm]{Question}
\newtheorem*{quesa}{Question A}
\newtheorem*{quesb}{Question B}
\newtheorem{case}{Case}
\newtheorem{setting}[thm]{Setting}
\newtheorem{rem}[thm]{Remark}
\newtheorem{note}[thm]{Note}
\newtheorem{remark}[thm]{Remark}

\theoremstyle{remark}
\newtheorem*{lemm}{Lemma}
\newtheorem*{pf}{{\sl Proof}}
\newtheorem*{pfpr}{Proof of Proposition \eqref{positive Hecke}}
\newtheorem*{tpf}{{\sl Proof of Theorem 1.1}}
\newtheorem*{cpf1}{{\sl Proof of Claim 1}}
\newtheorem*{cpf2}{{\sl Proof of Claim 2}}

\def\soc{\operatorname{Soc}}
\def\xx{\text{{\boldmath$x$}}}
\def\L{\mathrm{U}}
\def\Z{\mathcal{Z}}
\def\D{\mathcal{D}}
\def\Min{\operatorname{Min}}
\def\Gdim{\operatorname{Gdim}}
\def\pd{\operatorname{pd}}
\def\GCD{\operatorname{GCD}}
\def\Ext{\operatorname{Ext}}
\def\X{\mathcal{X}}
\def\S{\mathcal{S}}
\def\XX{\mathbb{X}}
\def\Im{\operatorname{Im}}
\def\s{\operatorname{s}}
\def\Ker{\operatorname{Ker}}
\def\KK{\mathbb{K}}
\def\bbZ{\mathbb{Z}}
\def\LL{\mathbb{U}}
\def\E{\operatorname{E}}
\def\H{\operatorname{H}}

\def\Hom{\operatorname{Hom}}
\def\Ext{\operatorname{Ext}}
\def\RHom{\mathrm{{\bf R}Hom}}
\def\Tor{\operatorname{Tor}}
\def\Max{\operatorname{Max}}
\def\Assh{\operatorname{Assh}}
\def\End{\mathrm{End}}
\def\rad{\mathrm{rad}}
\def\Soc{\mathrm{Soc}}
\def\Proj{\operatorname{Proj}}

\def\Mod{\mathrm{Mod}}
\def\mod{\mathrm{mod}}
\def\G{{\sf G}}
\def\T{{\sf T}}

\def\Coker{\mathrm{Coker}}
\def\im{\mathrm{im}\;}
\def\st{{}^{\ast}}
\def\rank{\mathrm{rank}}
\def\a{\mathfrak a}
\def\b{\mathfrak b}
\def\c{\mathfrak c}
\def\e{\mathrm{e}}
\def\f{\mathfrak{f}}
\def\m{\mathfrak m}
\def\n{\mathfrak n}
\def\p{\mathfrak p}
\def\q{\mathfrak q}
\def\P{\mathfrak P}
\def\Q{\mathfrak Q}
\def\N{\Bbb N}
\def\C{\mathcal{C}}
\def\K{\mathrm{K}}
\def\H{\mathrm{H}}
\def\J{\mathrm{J}}
\def\Gr{\mathrm G}
\def\FC{\mathrm F}
\def\Var{\mathrm V}
\def\V{\mathrm V}
\def\r{\mathrm{r}}

\newcommand{\nCM}{\mathrm{nCM}}
\newcommand{\nSCM}{\mathrm{nSCM}}
\newcommand{\Aut}{\mathrm{Aut}}
\newcommand{\Att}{\mathrm{Att}}
\newcommand{\Ann}{\mathrm{Ann}}

\newcommand{\rma}{\mathrm{a}}
\newcommand{\rmb}{\mathrm{b}}
\newcommand{\rmc}{\mathrm{c}}
\newcommand{\rmd}{\mathrm{d}}
\newcommand{\rme}{\mathrm{e}}
\newcommand{\rmf}{\mathrm{f}}
\newcommand{\rmg}{\mathrm{g}}
\newcommand{\rmh}{\mathrm{h}}
\newcommand{\rmi}{\mathrm{i}}
\newcommand{\rmj}{\mathrm{j}}
\newcommand{\rmk}{\mathrm{k}}
\newcommand{\rml}{\mathrm{l}}
\newcommand{\rmm}{\mathrm{m}}
\newcommand{\rmn}{\mathrm{n}}
\newcommand{\rmo}{\mathrm{o}}
\newcommand{\rmp}{\mathrm{p}}
\newcommand{\rmq}{\mathrm{q}}
\newcommand{\rmr}{\mathrm{r}}
\newcommand{\rms}{\mathrm{s}}
\newcommand{\rmt}{\mathrm{t}}
\newcommand{\rmu}{\mathrm{u}}
\newcommand{\rmv}{\mathrm{v}}
\newcommand{\rmw}{\mathrm{w}}
\newcommand{\rmx}{\mathrm{x}}
\newcommand{\rmy}{\mathrm{y}}
\newcommand{\rmz}{\mathrm{z}}

\newcommand{\rmA}{\mathrm{A}}
\newcommand{\rmB}{\mathrm{B}}
\newcommand{\rmC}{\mathrm{C}}
\newcommand{\rmD}{\mathrm{D}}
\newcommand{\rmE}{\mathrm{E}}
\newcommand{\rmF}{\mathrm{F}}
\newcommand{\rmG}{\mathrm{G}}
\newcommand{\rmH}{\mathrm{H}}
\newcommand{\rmI}{\mathrm{I}}
\newcommand{\rmJ}{\mathrm{J}}
\newcommand{\rmK}{\mathrm{K}}
\newcommand{\rmL}{\mathrm{L}}
\newcommand{\rmM}{\mathrm{M}}
\newcommand{\rmN}{\mathrm{N}}
\newcommand{\rmO}{\mathrm{O}}
\newcommand{\rmP}{\mathrm{P}}
\newcommand{\rmQ}{\mathrm{Q}}
\newcommand{\rmR}{\mathrm{R}}
\newcommand{\rmS}{\mathrm{S}}
\newcommand{\rmT}{\mathrm{T}}
\newcommand{\rmU}{\mathrm{U}}
\newcommand{\rmV}{\mathrm{V}}
\newcommand{\rmW}{\mathrm{W}}
\newcommand{\rmX}{\mathrm{X}}
\newcommand{\rmY}{\mathrm{Y}}
\newcommand{\rmZ}{\mathrm{Z}}

\newcommand{\cala}{\mathcal{a}}
\newcommand{\calb}{\mathcal{b}}
\newcommand{\calc}{\mathcal{c}}
\newcommand{\cald}{\mathcal{d}}
\newcommand{\cale}{\mathcal{e}}
\newcommand{\calf}{\mathcal{f}}
\newcommand{\calg}{\mathcal{g}}
\newcommand{\calh}{\mathcal{h}}
\newcommand{\cali}{\mathcal{i}}
\newcommand{\calj}{\mathcal{j}}
\newcommand{\calk}{\mathcal{k}}
\newcommand{\call}{\mathcal{l}}
\newcommand{\calm}{\mathcal{m}}
\newcommand{\caln}{\mathcal{n}}
\newcommand{\calo}{\mathcal{o}}
\newcommand{\calp}{\mathcal{p}}
\newcommand{\calq}{\mathcal{q}}
\newcommand{\calr}{\mathcal{r}}
\newcommand{\cals}{\mathcal{s}}
\newcommand{\calt}{\mathcal{t}}
\newcommand{\calu}{\mathcal{u}}
\newcommand{\calv}{\mathcal{v}}
\newcommand{\calw}{\mathcal{w}}
\newcommand{\calx}{\mathcal{x}}
\newcommand{\caly}{\mathcal{y}}
\newcommand{\calz}{\mathcal{z}}

\newcommand{\calA}{\mathcal{A}}
\newcommand{\calB}{\mathcal{B}}
\newcommand{\calC}{\mathcal{C}}
\newcommand{\calD}{\mathcal{D}}
\newcommand{\calE}{\mathcal{E}}
\newcommand{\calF}{\mathcal{F}}
\newcommand{\calG}{\mathcal{G}}
\newcommand{\calH}{\mathcal{H}}
\newcommand{\calI}{\mathcal{I}}
\newcommand{\calJ}{\mathcal{J}}
\newcommand{\calK}{\mathcal{K}}
\newcommand{\calL}{\mathcal{L}}
\newcommand{\calM}{\mathcal{M}}
\newcommand{\calN}{\mathcal{N}}
\newcommand{\calO}{\mathcal{O}}
\newcommand{\calP}{\mathcal{P}}
\newcommand{\calQ}{\mathcal{Q}}
\newcommand{\calR}{\mathcal{R}}
\newcommand{\calS}{\mathcal{S}}
\newcommand{\calT}{\mathcal{T}}
\newcommand{\calU}{\mathcal{U}}
\newcommand{\calV}{\mathcal{V}}
\newcommand{\calW}{\mathcal{W}}
\newcommand{\calX}{\mathcal{X}}
\newcommand{\calY}{\mathcal{Y}}
\newcommand{\calZ}{\mathcal{Z}}

\newcommand{\fka}{\mathfrak{a}}
\newcommand{\fkb}{\mathfrak{b}}
\newcommand{\fkc}{\mathfrak{c}}
\newcommand{\fkd}{\mathfrak{d}}
\newcommand{\fke}{\mathfrak{e}}
\newcommand{\fkf}{\mathfrak{f}}
\newcommand{\fkg}{\mathfrak{g}}
\newcommand{\fkh}{\mathfrak{h}}
\newcommand{\fki}{\mathfrak{i}}
\newcommand{\fkj}{\mathfrak{j}}
\newcommand{\fkk}{\mathfrak{k}}
\newcommand{\fkl}{\mathfrak{l}}
\newcommand{\fkm}{\mathfrak{m}}
\newcommand{\fkn}{\mathfrak{n}}
\newcommand{\fko}{\mathfrak{o}}
\newcommand{\fkp}{\mathfrak{p}}
\newcommand{\fkq}{\mathfrak{q}}
\newcommand{\fkr}{\mathfrak{r}}
\newcommand{\fks}{\mathfrak{s}}
\newcommand{\fkt}{\mathfrak{t}}
\newcommand{\fku}{\mathfrak{u}}
\newcommand{\fkv}{\mathfrak{v}}
\newcommand{\fkw}{\mathfrak{w}}
\newcommand{\fkx}{\mathfrak{x}}
\newcommand{\fky}{\mathfrak{y}}
\newcommand{\fkz}{\mathfrak{z}}

\newcommand{\fkA}{\mathfrak{A}}
\newcommand{\fkB}{\mathfrak{B}}
\newcommand{\fkC}{\mathfrak{C}}
\newcommand{\fkD}{\mathfrak{D}}
\newcommand{\fkE}{\mathfrak{E}}
\newcommand{\fkF}{\mathfrak{F}}
\newcommand{\fkG}{\mathfrak{G}}
\newcommand{\fkH}{\mathfrak{H}}
\newcommand{\fkI}{\mathfrak{I}}
\newcommand{\fkJ}{\mathfrak{J}}
\newcommand{\fkK}{\mathfrak{K}}
\newcommand{\fkL}{\mathfrak{L}}
\newcommand{\fkM}{\mathfrak{M}}
\newcommand{\fkN}{\mathfrak{N}}
\newcommand{\fkO}{\mathfrak{O}}
\newcommand{\fkP}{\mathfrak{P}}
\newcommand{\fkQ}{\mathfrak{Q}}
\newcommand{\fkR}{\mathfrak{R}}
\newcommand{\fkS}{\mathfrak{S}}
\newcommand{\fkT}{\mathfrak{T}}
\newcommand{\fkU}{\mathfrak{U}}
\newcommand{\fkV}{\mathfrak{V}}
\newcommand{\fkW}{\mathfrak{W}}
\newcommand{\fkX}{\mathfrak{X}}
\newcommand{\fkY}{\mathfrak{Y}}
\newcommand{\fkZ}{\mathfrak{Z}}
\newcommand{\kr}{\mathrm{K}_R}
\newcommand{\mapright}[1]{%
\smash{\mathop{%
\hbox to 1cm{\rightarrowfill}}\limits^{#1}}}

\newcommand{\mapleft}[1]{%
\smash{\mathop{%
\hbox to 1cm{\leftarrowfill}}\limits_{#1}}}

\newcommand{\mapdown}[1]{\Big\downarrow
\llap{$\vcenter{\hbox{$\scriptstyle#1\,$}}$ }}
\renewcommand{\o}[1]{{#1}^{\circ}}

\newcommand{\mapup}[1]{\Big\uparrow
\rlap{$\vcenter{\hbox{$\scriptstyle#1\,$}}$ }}
\def\grade{\operatorname{grade}}
\def\depth{\operatorname{depth}}
\def\AGL{\operatorname{AGL}}
\def\Supp{\operatorname{Supp}}
\def\ann{\operatorname{Ann}}
\def\Ass{\operatorname{Ass}}
\def\height{\mathrm{ht}}
\def\Sp{\operatorname{Spec}}
\def\Spf{\operatorname{Spf}}
\def\Syz{\mathrm{Syz}}
\def\hdeg{\operatorname{hdeg}}
\def\id{\operatorname{id}}
\def\gr{\mbox{\rm gr}}
\def\Zdv{\operatorname{Zdv}}
\def\Wd{\operatorname{\mathrm{WD}}}
\def\r{\mathrm{r}}

\def\A{{\mathcal A}}
\def\B{{\mathcal B}}
\def\F{{\mathcal F}}
\def\Y{{\mathcal Y}}
\def\W{{\mathcal W}}
\def\M{{\mathcal M}}
\def\O{{\mathcal O}}
\def\H{{\mathcal H}}
\def\G{{\mathcal G}}
\def\R{{\mathcalR}}
\def\I{{\mathcal I}}
\def\J{{\mathcal J}}
\def\L{{\mathcal L}}
\def\U{{\mathcal U}}
\def\fM{\mathfrak M}
\def\fN{\mathfrak N}
\def\fl{\pi^\flat}
\def\E{{\mathcal E}}
\renewcommand{\R}{\mathcal{R}}

\def\yy{\text{\boldmath $y$}}
\def\rb{\overline{R}}
\def\ol{\overline}
\def\Deg{\mathrm{deg}}

\def\PP{\mathbb{P}}
\def\II{\mathbb{I}}
\newcommand{\w}{\wedge}
\newcommand{\dep}{\mathrm{depth}}
\newcommand{\kz}[1]{\mathrm{K}_{\bullet}(#1)}
\newcommand{\tc}[2]{\langle #1, #2 \rangle}
\newcommand{\hm}[2]{\operatorname{Hom}_{R}(#1, #2 )}
\newcommand{\pt}{\partial}
\newcommand{\ckz}[1]{\mathrm{K}_{\bullet}(#1)}
\newcommand{\loc}[1]{\mathrm{H}_{\m}^d({#1})}
\newcommand{\hy}[2]{\mathrm{H}_{\m}^{#1}(#2)}
\newcommand{\cy}[2]{\mathrm{H}_{#1}(#2)}
\newcommand{\lc}[3]{\mathrm{H}_{#1}^{#2}(#3)}
\newcommand{\exe}[2]{\operatorname{Ext}_{R}^{#1}(#2,\kr)}
\newcommand{\ext}[3]{\operatorname{Ext}_{R}^{#1}(#2, #3)}
\newcommand{\ie}[1]{\mathrm{E}_{R}(#1)}
\newcommand{\vin}{\rotatebox[origin=c]{90}{$\in$}}
\newcommand{\ts}{\otimes}
\newcommand{\as}[1]{\operatorname{Ass}_{R} #1}
\newcommand{\mcF}{\mathcal{F}}
\newcommand{\idd}[1]{\mathrm{id}_{R}({#1})}
\newcommand{\wrr}{\widehat{R}}
\newcommand{\mi}[1]{\mu^i(\m,{#1})}
\newcommand{\krr}[1]{\mathrm{K}_{#1}}
\newcommand{\rt}[1]{\mathrm{r}_R({#1})}
\newcommand{\ci}{\mathbb{C}_{i}(R)}
\newcommand{\ccr}[1]{\mathbb{C}_{#1}(R)}
\newcommand{\dg}{\textsuperscript{\textdagger}}
\newcommand{\rr}[2]{\mathrm{r}_{#1}(#2)}
\newcommand{\wm}{\widehat{\m}}
\newcommand{\wE}{\E_{\wrr}(\wrr/\wm)}
\renewcommand{\sp}{\mathrm{Spec}_S(\mathcal{A})}
\newcommand{\proj}{\mathrm{Proj}}

\newcommand{\hms}[2]{\mathrm{Hom}_S(#1,#2)}
\newcommand{\hyy}[3]{\mathrm{H}_{\underline{#1}}^{#2}(#3)}
\newcommand{\rmod}{R-\mathrm{Mod}}
\newcommand{\tth}{\text{th}}
\newcommand{\asl}[1]{\operatorname{Ass}_{S^{-1}R}S^{-1} #1}
\newcommand{\uijk}{U_{ij\lambda}}
\renewcommand{\Proj}{\mathrm{Proj}}
\renewcommand{\rad}{\mathrm{rad}}
\newcommand{\norm}[1]{\|#1\|}
\newcommand{\fil}[1]{\mathrm{Fil}^{#1}}
\newcommand{\Ql}{\overline{\mathbf{Q}}_{\ell}}
\newcommand{\GL}{\mathrm{GL}_n(\overline{\mathbf{Q}}_{\ell})}
\newcommand{\un}{F^{\mathrm{unr}}}
\newcommand{\tr}{F^{\mathrm{tr}}}
\newcommand{\tl}{t_{\ell}}
\newcommand{\gl}{\mathrm{GL}}
\newcommand{\en}{\mathrm{End}}
\newcommand{\spe}{\mathrm{SL}_2(\mathbb{F}_p)}
\newcommand{\gen}{\mathrm{GL}_2(\mathbb{F}_p)}
\newcommand{\glr}{\mathrm{GL}_n(\mathbb{R})}
\newcommand{\supp}{\mathrm{supp}\;}
\newcommand{\Qbar}{\overline{\mathbf{Q}}}
\newcommand{\Zbar}{\overline{\mathbb{Z}}_p}
\newcommand{\oks}{\O_{K,S}}
\newcommand{\cl}{\mathrm{Cl}}
\newcommand{\cls}{\mathrm{cl}}
\newcommand{\ord}{\mathrm{ord}}
\newcommand{\ab}{\mathrm{ab}}
\newcommand{\unr}{\mathrm{unr}}
\newcommand{\gal}{\mathrm{Gal}}
\newcommand{\vol}{\mathrm{vol}}
\newcommand{\reg}{\mathrm{Reg}}
\newcommand{\re}{\mathrm{Re}}
\newcommand{\spm}{\mathrm{Spm}}
\newcommand{\kunr}{K^\mathrm{unr}}
\newcommand{\kt}{K^\mathrm{t}}
\newcommand{\pnm}{\mathbb{P}^n_R\times_R\mathbb{P}^m_R}
\newcommand{\pnmr}{\mathbb{P}_R^{nm+n+m}}
\newcommand{\bk}{\mathrm{Breuil}\textendash\mathrm{Kisin}}
\newcommand{\bkf}{\mathrm{Breuil}\textendash\mathrm{Kisin}\textendash\mathrm{Fargues}}
\renewcommand{\ss}{\mathrm{ss}}
\newcommand{\crys}{\mathrm{crys}}
\newcommand{\Ao}{A^{\circ}}
\newcommand{\Ddr}{D_{\mathrm{dR}}(\fM_{\Ao}^{\inf})}
\newcommand{\und}[1]{\underline{#1}}
\newcommand{\xdss}{\X_d^{\ss,\underline{\lambda},\tau}}
\newcommand{\xdcrys}{\X_d^{\crys,\underline{\lambda},\tau}}
\newcommand{\BK}{\mathrm{BK}}
\newcommand{\disc}{\mathrm{disc}}
\newcommand{\kbas}{K^\mathrm{basic}}
\newcommand{\cyc}{\mathrm{cyc}}
\newcommand{\aka}{\mathbf{A}_{K,A}}

\newcommand{\ak}[1]{\mathbf{A}_{K,#1}}
\newcommand{\lt}{\mathrm{LT}}
\newcommand{\xan}{X^{\mathrm{an}}}
\newcommand{\ur}[1]{\mathrm{ur}_{#1}}
\title{\large\bf The Emerton--Gee stack of rank one $(\varphi,\Gamma)$-modules}
\author{{\sc\normalsize Dat Pham}}
\date{}
\maketitle
\begin{abstract}
    We give a classification of rank one $(\varphi,\Gamma)$-modules with coefficients in a $p$-adically complete $\mathbf{Z}_p$-algebra. As a consequence, we obtain a new proof of \cite[Prop. 7.2.17]{EG22}, which gives an explicit description of the Emerton--Gee stack of $(\varphi,\Gamma)$-modules in the rank one case. In fact, our method also applies in the context of rank one \' etale $\varphi$-modules (i.e. in the absence of a $\Gamma$-action), generalizing another result of Emerton--Gee. 
\end{abstract}
\setcounter{tocdepth}{1}

\tableofcontents
\section{Introduction}
For simplicity, we fix our local field to be $\mathbf{Q}_p$ in this introduction. Let $\mathbf{A}'_{\mathbf{Q}_p}$ be the $p$-adic completion of the Laurent series ring $\mathbf{Z}_p((T))$, endowed with the usual commuting semilinear actions of $\varphi$ and $\Gamma=\mathbf{Z}_p^\times$. An \' etale $(\varphi,\Gamma)$-module is, by definition, a finite $\mathbf{A}'_{\mathbf{Q}_p}$-module endowed
with commuting semilinear actions of $\varphi$ and $\Gamma$, with the property that the linearized action of $\varphi$ is an isomorphism. The most important feature of \' etale $(\varphi,\Gamma)$-modules is that they are naturally equivalent to continuous representations of $G_{\mathbf{Q}_p}$ on finite $\mathbf{Z}_p$-modules (cf. \cite{Fon90}). 

In \cite{EG22}, Emerton and Gee define and study moduli stacks parametrizing families of \' etale $(\varphi,\Gamma)$-modules. More specifically, they consider the stack $\X_d$ over $\Spf\mathbf{Z}_p$ whose groupoid of $A$-valued points, for any $p$-adically complete $\mathbf{Z}_p$-algebra $A$, is given by the groupoid of rank $d$ projective \' etale $(\varphi,\Gamma)$-modules over $\mathbf{A}'_{\mathbf{Q}_p,A}:=\mathbf{A}'_{\mathbf{Q}_p}\widehat{\otimes}_{\mathbf{Z}_p}A$. The geometry of the stack $\X_d$ has been studied extensively in \cite{EG22}. In particular, the authors show that $\X_d$ is a Noetherian formal algebraic stack, and moreover, its underlying reduced substack is an algebraic stack of finite type over $\mathbf{F}_p$, whose irreducible components admit a natural labelling by Serre weights. 

The goal of this note is to prove the following classification of families of rank one \' etale $(\varphi,\Gamma)$-modules.
\begin{thm}[Theorem \eqref{main prop}]
Let $A$ be a $p$-adically complete $\mathbf{Z}_p$-algebra. Let $M$ be a rank one \' etale $(\varphi,\Gamma)$-module with $A$-coefficients. Then $M$ has the form $\mathbf{A}'_{\mathbf{Q}_p,A}(\delta)\otimes_AL$ for some character $\delta: \mathbf{Q}_p^\times\to A^\times$ and some invertible $A$-module $L$. Here, $\mathbf{A}'_{\mathbf{Q}_p,A}(\delta)$ denotes the free $(\varphi,\Gamma)$-module of rank 1 with a basis $v$
for which $\varphi(v)=\delta(p)v$ and $\gamma(v)=\delta(\gamma)v$ for $\gamma\in\mathbf{Z}_p^\times$.
\end{thm}
As a consequence, we deduce the following explicit description of the stack $\X_1$.
\begin{cor}[Corollary \eqref{X 1 iso}]
There is an isomorphism
\begin{displaymath}
\left[\left(\Spf \mathbf{Z}_p[[\mathbf{Z}_p^\times]]\times\widehat{\mathbf{G}}_m\right)/\widehat{\mathbf{G}}_m\right]\xrightarrow{\sim} \X_1,
\end{displaymath}
where $\widehat{\mathbf{G}}_m$ denotes the $p$-adic completion of $\widehat{\mathbf{G}}_{m,\mathbf{Z}_p}$, and in the formation of the quotient stack, the action of $\widehat{\mathbf{G}}_m$ is taken to be trivial. 
\end{cor} 
We emphasize that the above description is already given in \cite[Prop. 7.2.17]{EG22}. However, our argument here is different; in particular it avoids the use of uniform bounds on the ramification of families of characters valued in finite Artinian algebras (see \cite[\textsection 7.3]{EG22}). 

\begin{remark}
In fact, our method also applies to give explicit descriptions of the stacks of rank one \' etale $\varphi$-modules (i.e. in the absence of a $\Gamma$-action), generalizing \cite[Prop. 7.2.11]{EG22} to a large class of coefficient rings. See Subsection \ref{etale phi description}.
\end{remark}
\begin{notation}
We mostly follow the notation in \cite{EG22}. In particular, we fix a finite extension $K/\mathbf{Q}_p$ with residue field $k$ and inertia degree $f$. Fix also an algebraic closure $\overline{K}$ of $K$, with absolute Galois group $G_K$, Weil group $W_K$, and inertia group $I_K$. As usual, $W_K^{\mathrm{ab}}$ denotes the abelianization of $W_K$, while $I_K^{\mathrm{ab}}$ denotes the image of $I_K$ in $W_K^{\mathrm{ab}}$. We denote by $\mathbf{C}^\flat$ the tilt of the completion $\mathbf{C}:=\widehat{\overline{K}}$, by $K_{\mathrm{cyc}}$ the cyclotomic $\mathbf{Z}_p$-extension of $K$ and by $k_{\infty}$ its residue field. We also fix a finite extension $E/\mathbf{Q}_p$ with ring of integers $\O$, which will serve as the base of our coefficients. As usual, $\varpi$ (resp. $\mathbf{F}$) denotes a uniformizer (resp. the residue field) of $\O$. We will fix throughout an embedding $k\hookrightarrow \mathbf{F}$.

We refer the reader to \cite[\textsection 2.2]{EG22} for the definition of the coefficient rings $\ak{A}$ of our $(\varphi,\Gamma)$-modules. Finally, as the field $K$ is fixed throughout, we will often drop $K$ from the notation in what follows. 
\end{notation}
\begin{acknowledgements}
I am grateful to Bao V. Le Hung and Stefano Morra for their encouragement and for various helpful discussions. I would also like to thank Matthew Emerton and Toby Gee for pointing out a mistake in a previous version, as well as the referee for some useful suggestions which help improve the readability of the paper. This project has received funding from the European Union’s Horizon 2020 research and innovation programme under the Marie Sk\l odowska-Curie grant agreement No. 945322.
\end{acknowledgements}
\section{\texorpdfstring{$(\varphi,\Gamma)$}{(phi,Gamma)}-modules associated to characters of the Weil group}
In this section, we explain how to associate a free \' etale $(\varphi,\Gamma)$-module of rank 1 to any character of $W_K$. 

First recall the following result of Dee, which is a generalization of Fontaine's equivalence between Galois representations on finite $\mathbf{Z}_p$-modules and \' etale $(\varphi,\Gamma)$-modules to the context with coefficients. For simplicity, we only state the result for Artinian coefficients. 
\begin{thm}[\cite{Dee01}]\label{Dee}
Let $A$ be a finite Artinian local $\O$-algebra, and let $W(\mathbf{C}^\flat)_A:=W(\mathbf{C}^\flat)\otimes_{\mathbf{Z}_p}A$. Then the functor 
\begin{displaymath}
M\mapsto T_A(M):=(W(\mathbf{C}^\flat)_A\otimes_{\aka}M)^{\varphi=1}
\end{displaymath}
defines an equivalence between the category of finite projective \' etale $(\varphi,\Gamma)$-modules with $A$-coefficients, and the category of finite free $A$-modules with a continuous action of $G_K$.
\end{thm}
We want to extend the above construction of rank one \' etale $(\varphi,\Gamma)$-modules from Galois characters to the case where $A$ is an arbitrary $\varpi$-adically complete $\O$-algebra. We begin with the case of unramified characters.   
\begin{lem}\label{universal unramifed char lem}
Let $A$ be an $\O$-algebra, and let $a\in A^\times$. Then, up to isomorphism, there is a unique free \' etale $\varphi$-module  $D_{k,a}$ of rank one over $W(k)\otimes_{\mathbf{Z}_p}A$ with the property that $\varphi^f=1\otimes a$ on $D_{k,a}$.
\end{lem}
\begin{proof}
We need to show that the norm map $(W(k)\otimes_{\mathbf{Z}_p}A)^\times\to A^\times, x\mapsto N(x):=x\varphi(x)\ldots \varphi^{f-1}(x)$ is surjective with kernel given by the set $\{\varphi(y)/y\;|\;y\in (W(k)\otimes_{\mathbf{Z}_p}A)^\times\}$. Since $\mathbf{F}$ is assumed to contain $k$, $\O$ (and hence $A$) is naturally a $W(k)$-algebra. In particular, we have an $A$-algebra isomorphism $W(k)\otimes_{\mathbf{Z}_p}A\to \prod A, x\otimes 1\mapsto (x,\varphi(x),\ldots,\varphi^{f-1}(x))$. The lemma then follows easily using this isomorphism. 
\end{proof}
\begin{defn}
Let $A$ be a $\varpi$-adically complete $\O$-algebra, and let $a\in A^\times$. Define $\aka(\ur{a}):=D_{k,a}\otimes_{W(k)\otimes_{\mathbf{Z}_p}A}\ak{A}$. This is a rank one \' etale $(\varphi,\Gamma)$-module with $A$-coefficients, where we let $\varphi$ act diagonally, and $\Gamma$ act on the second factor. 
\end{defn}
\begin{lem}\label{universal unramified}
Let $A$ be a finite Artinian local $\O$-algebra, and let $a\in A^\times$. Then, under Dee's equivalence \eqref{Dee}, $\aka(\ur{a})$ corresponds to the unramified character $\ur{a}$ of $G_K$ sending geometric Frobenii to $a$.
\end{lem}
\begin{proof}
By definition, the rank one $A$-representation of $G_{K}$ corresponding to $\aka(\ur{a})$ is given by
\begin{align*}
    V &:= (W(\mathbf{C}^\flat)_A\otimes_{\aka}\aka(\ur{a}))^{\varphi=1}\\
    &\cong \{hv\;|\; h\in W(\mathbf{C}^\flat)\otimes_{\mathbf{Z}_p}A\;\text{such that $\varphi(hv)=hv$}\},
\end{align*}
where $v$ is a basis of $D_{k,a}$. Assume $V$ has a basis $hv$ with $h\in W(\overline{\mathbf{F}}_p)\otimes_{\mathbf{Z}_p}A$. We verify that $G_{K}$ acts on this basis via the unramified character taking geometric Frobenii to $a$. First, for $\sigma\in I_K$, we have $\sigma(hv)=\sigma(h)v=hv$ (note that $\sigma(h)=h$ as $h\in W(\overline{\mathbf{F}}_p)\otimes_{\mathbf{Z}_p}A$). Now, let $\sigma=\varphi_q^{-1}$ be an arithmetic Frobenius. From the relation $\varphi(hv)=hv$ and the fact that $\varphi^f(v)=av$, we obtain $\sigma(h)av=\varphi^f(hv)=hv$ whence $\sigma(hv)=\sigma(h)v=a^{-1}(hv)$, as desired. 

It remains to find a basis as stated. If $A$ is a field, say $A=\mathbf{F}_q$ for some finite extension $\mathbf{F}_q/\mathbf{F}$, this can be done by using the ring isomorphism $\mathbf{C}^\flat\otimes_{\mathbf{F}_p}\mathbf{F}_q\xrightarrow{\sim} \prod \mathbf{C}^\flat$.
Indeed, $h$ is a vector in $\prod \mathbf{C}^\flat$ whose coordinates satisfy a finite set of polynomial equations with coefficients in $\overline{\mathbf{F}}_p$, so it necessarily lies in $\prod\overline{\mathbf{F}}_p$. In general, by factoring $A\twoheadrightarrow A/\m_A$ as a chain of square-zero thickenings, we may assume that, for some ideal $I$ with $I^2=0$, there is a basis $\overline{h}v$ of $V/I V$ with $\overline{h}\in W(\overline{\mathbf{F}}_p)\otimes_{\mathbf{Z}_p}(A/I)$. Let $h\in W(\overline{\mathbf{F}}_p)\otimes_{\mathbf{Z}_p}A$ be a lift of $\overline{h}$. Then $\varphi(hv)=g(hv)$ for some $g\in 1+W(\overline{\mathbf{F}}_p)\otimes I$, say $g=1+g_1\otimes m_1+\ldots+g_n\otimes m_n$ with $g_i\in W(\overline{\mathbf{F}}_p)$ and $m_i\in I$. For each $i$, choose $h_i\in W(\overline{\mathbf{F}}_p)$ such that $\varphi(h_i)-h_i=-g_i$. Let $f:=1+h_1\otimes m_1+\ldots h_n\otimes m_n\in 1+W(\overline{\mathbf{F}}_p)\otimes I$. Then $\varphi(f)=f/g$, and hence $\varphi(hfv)=hfv$. Finally, as $(hf)v$ lifts a basis of $V/I V$, it is a basis of $V$ by Nakayama's lemma.
\end{proof}
Recall that $\widehat{\mathbf{G}}_m$ denotes the $\varpi$-adic completion of $\mathbf{G}_{m,\O}$. We denote the resulting map $\widehat{\mathbf{G}}_m\to \X_1, a \mapsto \aka(\ur{a})$ by $\ur{x}$ (where we can think of $x$ as the coordinate on $\widehat{\mathbf{G}}_m=\Spf \O[x,x^{-1}]$), and refer to it simply as the universal unramified character\footnote{Note that the definition given in \cite[\textsection 5.3]{EG22} of the map $\ur{x}$ is slightly incorrect in the sense that it does not agree with the construction of Dee for Artinian coefficients.}. 

We now consider the case of a general character of the Weil group $W_K$. It is convenient to introduce some notation.
\begin{defn}
Let $\xan$ be the functor on $\varpi$-adically complete $\O$-algebras taking $A$ to the set of (continuous) characters $\delta: W_K\to A^\times$. 
\end{defn}
As a first remark, we note that fixing a geometric Frobenius $\sigma\in G_K$ (or equivalently, an isomorphism $W_K^{\mathrm{ab}}\cong I_K^{\mathrm{ab}}\times\mathbf{Z}$) is equivalent to fixing an isomorphism of (Noetherian) affine  formal schemes
\begin{displaymath}
    \xan\simeq \Spf\O[[I_K^{\mathrm{ab}}]]\times \widehat{\mathbf{G}}_{m} 
\end{displaymath}
over $\Spf \O$. Concretely, at the level of $A$-valued points (with $A$ a $\varpi$-adically complete $\O$-algebra), this is given by the assignment $(\delta: W_K\to A^\times)\mapsto (\delta|_{I_K^{\mathrm{ab}}},\delta(\sigma))$. 

In what follows, we will always endow $\xan$ with the trivial action of $\widehat{\mathbf{G}}_m$. 
\begin{lem}\label{delta-->aka(delta)}
There is a morphism $\xan\to \X_1, \delta\mapsto \aka(\delta)$ with the property that for any finite Artinian $\O$-algebra $A$, the $(\varphi,\Gamma)$-module $\aka(\delta)$ corresponds, under Dee's equivalence \eqref{Dee}, to the character $\delta$.
\end{lem}
\begin{proof}
Fix a geometric Frobenius $\sigma\in G_K$, and hence an isomorphism $\xan\cong \Spf \O[[I_K^{\mathrm{ab}}]]\times\widehat{\mathbf{G}}_m$. For each finite Artinian quotient $A$ of $\O[[I_K^{\mathrm{ab}}]]$, we extend the map $I_K^{\mathrm{ab}}\to A^\times$ to a character $G_K\to A^\times$ by taking $\sigma$ to $1$. Under Dee's equivalence \eqref{Dee}, this gives rise to a rank one \' etale $(\varphi,\Gamma)$-module with $A$-coefficients, i.e. an object of $\X_1(A)$. As $\O[[I_K^{\mathrm{ab}}]]$ is the inverse limit of all such quotients $A$, we obtain a map $\Spf \O[[I_K^{\mathrm{ab}}]]\to \X_1$. We can now define $\xan\to \X_1$ as the composite 
\begin{displaymath}
\xan\cong \Spf \O[[I_K^{\mathrm{ab}}]]\times\widehat{\mathbf{\mathbf{G}}}_m\to \X_1\times \widehat{\mathbf{G}}_m\to \X_1,    
\end{displaymath}
where the last map is given by taking tensor product with the universal unramified character $\ur{x}$. That the map is compatible with Dee's equivalence in case of Artinian coefficients follows immediately from construction and Lemma \eqref{universal unramified}.
\end{proof}
It is easy to check that the construction $\delta\mapsto \aka(\delta)$ is independent on our choice of $\sigma$ (see also Remark \eqref{independent of frob} below), and that $\aka(\delta_1\delta_2)\cong \aka(\delta_1)\otimes_{\aka}\aka(\delta_2)$ for all $\delta_i$. Of course, in case $\delta=\ur{a}$ is an unramified character, this agrees with the notation introduced earlier. For a $(\varphi,\Gamma)$-module $M$ with $A$-coefficients, we set $M(\delta):=M\otimes_{\aka}\aka(\delta)$, equipped with the obvious (diagonal) $(\varphi,\Gamma)$-structure.
\section{Explicit descriptions of the rank one stacks}
\subsection{Statement}
Our main result in this text is the following.
\begin{thm}\label{main prop}
Let $A$ be a $\varpi$-adically complete $\O$-algebra. Let $M$ be a rank one \' etale $(\varphi,\Gamma)$-module with $A$-coefficients. Then there exist a unique continuous character $\delta: W_K\to A^\times$ and a unique (up to isomorphism) invertible $A$-module $L$, such that $M\cong \aka(\delta)\otimes_AL$.
\end{thm}
\begin{cor}[{{\cite[Prop. 7.2.17]{EG22}}}]\label{X 1 iso}
The map $\xan\to \X_1, \delta\mapsto \aka(\delta)$ induces an isomorphism 
\begin{displaymath}
    \left[\xan/\widehat{\mathbf{G}}_m\right]\xrightarrow{\sim} \X_1.
\end{displaymath} 
\end{cor}
\begin{proof}[Proof of Corollary \eqref{X 1 iso}]
Since $\xan$ is endowed with the trivial action of $\widehat{\mathbf{G}}_m$, the quotient stack $[\xan/\widehat{\mathbf{G}}_m]$ is naturally identified with $[\Spf \O/\widehat{\mathbf{G}}_m]\times_{\Spf \O}\xan$. In other words, for any $\varpi$-adically complete $\O$-algebra $A$, its groupoid of $A$-valued points is equivalent to the groupoid of pairs $(L,\delta)$ consisting of an invertible $A$-module $L$, and a character $\delta\in \xan(A)$. Via this identification, the map $\xan\to \X_1$ factors through the map
\begin{align*}
    \left[\xan/\widehat{\mathbf{G}}_m\right] &\to \X_1
\end{align*}
defined by $(L,\delta)\mapsto \aka(\delta)\otimes_AL$. The result now follows from Theorem \eqref{main prop}, and the fact that the automorphism group of any rank one \' etale $(\varphi,\Gamma)$-module is given simply by the scalars:
\begin{displaymath}
\mathrm{Aut}_{\aka,\varphi,\Gamma}(M)=((M\otimes M^\vee)^{\varphi,\Gamma=1})^\times=(\aka^{\varphi,\Gamma=1})^\times=A^\times;
\end{displaymath}
here we have used \cite[Lem. 2.2.19 and Prop. 2.2.12]{EG22} for the last equality.
\end{proof}
\subsection{Proof}
This subsection is devoted to proving Theorem \eqref{main prop}. In order to streamline the arguments, we postpone the proof of one key result (Lemma \eqref{closed immersion crucial}) to $\mathsection$\ref{proof of crucial} below.

We begin with the uniqueness statement.
\begin{lem}\label{uniqueness}
The uniqueness part of Theorem \eqref{main prop} holds.
\end{lem}
\begin{proof}
As $(\aka)^{\varphi=1}=A$, we necessarily have $L\cong M(\delta^{-1})^{\varphi=1}$. It remains to show that if $\aka(\delta)\cong \aka(\delta')$ as $(\varphi,\Gamma)$-modules, then $\delta=\delta'$. Using that $\xan$ and $\X_1$ are both limit preserving (this is easy for $\xan$ by its definition; for $\X_1$, see \cite[Lem. 3.2.19]{EG22}), we may assume that $A$ is a finite type $\O/\varpi^a$-algebra for some $a\geq 1$. In this case, $A$ embeds naturally into the product of its Artinian quotients, and so we may assume further that $A$ is a finite Artinian $\O$-algebra. The lemma now follows since the map $\delta\mapsto \aka(\delta)$ recovers the equivalence between rank one \' etale $(\varphi,\Gamma)$-modules and Galois characters for Artinian coefficients (Lemma \eqref{delta-->aka(delta)}).
\end{proof}
\begin{remark}\label{monomorphism}
By Lemma \eqref{uniqueness} and the fact that the automorphism group of any rank one $(\varphi,\Gamma)$ module is simply $\widehat{\mathbf{G}}_m$, we see that the map $[\xan/\widehat{\mathbf{G}}_m]\hookrightarrow\X_1$ is at least a monomorphism. Showing that it is in fact essentially
surjective (i.e. an isomorphism) is equivalent to showing the existence part of Theorem \eqref{main prop}.
\end{remark}
The next lemma allows us to reduce to the case where our test object $A$ is a reduced $\mathbf{F}$-algebra. 
\begin{lem}\label{reduced is enough}
If Theorem \eqref{main prop} holds for reduced finite type $\mathbf{F}$-algebras $A$, then it holds for any $\varpi$-adically complete $\O$-algebra $A$.
\end{lem}
\begin{proof}
Let $A$ be an $\O/\varpi^a$-algebra for some $a\geq 1$, and let $M$ be a rank one \' etale $(\varphi,\Gamma)$-module with $A$-coefficients. We want to show that $M\cong \aka(\delta)\otimes_A L$ for some $\delta$ and $L$. As $\X_1$ is limit preserving, we may assume $A$ is of finite type over $\O/\varpi^a$. We will induct on the nilpotency index $e$ of the nilradical $A^{\circ\circ}$. The case $e=1$ is just our assumption. Assume now that $e\geq 2$. Let $I:=(A^{\circ\circ})^{e-1}$ and $\bar{A}:=A/I$. By the induction hypothesis, $M_{\bar{A}}:=M\otimes_A\bar{A}$ has the form $\ak{\bar{A}}(\bar{\delta})\otimes_{\bar{A}}\tilde{L}$ for some character $\overline{\delta}$ and some invertible $\bar{A}$-module $\tilde{L}$. Lifting $\bar{\delta}$ to a character $\delta\in\xan(A)$, $\tilde{L}$ to an invertible $A$-module $L$ (recall that finite projective modules always deform uniquely through nilpotent thickenings), and replacing $M$ with $M(\delta^{-1})\otimes_AL^\vee$, we may assume that $M_{\bar{A}}$ is trivial. By \cite[Prop. 5.1.33]{EG22}, the set of isomorphism classes of such $M$ is given by $H^1(\C^{\bullet}(\ak{I}))$, where $\C^\bullet(\ak{I})$ is the Herr complex of $\ak{I}:=I\ak{A}\cong (W(k_\infty)\otimes_{\mathbf{Z}_p}I)((t))$. Namely, given such $M$, there is an $\aka$-basis $v$ of $M$ so that $\varphi(v)=fv$ and $\gamma(v)=gv$ for some $f,g\in \ker((\aka)^\times\to (\ak{\bar{A}})^\times)=1+\ak{I}$, where $\gamma$ is a fixed topological generator of $\Gamma_K\cong\mathbf{Z}_p$. The corresponding cohomology class is then given by $[(f-1,g-1)]$.  

Let $(\O/\varpi^a)[I]$ be the usual square-zero thickening defined by $I$. Using the above description in terms of $H^1(\C^\bullet(\ak{I}))$, we see that $M$ arises as the base change of some rank one $(\varphi,\Gamma)$-module with $(\O/\varpi^a)[I]$-coefficients via the natural map $(\O/\varpi^a)[I]\to A$. Thus we may reduce to the case $A=(\O/\varpi^a)[I]$. By writing $I$ as the colimit of its finite sub-$\O$-modules and using again the fact that $\X_1$ is limit preserving, we may assume further that $I$ is finite over $\O$. But in this case $(\O/\varpi^a)[I]$ is a finite Artinian $\O$-algebra, so we are done by using (again) the fact that the construction $\delta\mapsto \aka(\delta)$ recovers the equivalence between Galois representations and $(\varphi,\Gamma)$-modules for Artinian coefficients.
\end{proof}
\begin{lem}\label{closed immersion crucial}
The map $\ur{x}: \widehat{\mathbf{G}}_m\to \X_1$ induces a closed immersion $[\widehat{\mathbf{G}}_m/\widehat{\mathbf{G}}_m]\hookrightarrow \X_1$.
\end{lem}
The proof of Lemma \eqref{closed immersion crucial} takes up $\mathsection$\ref{proof of crucial} below. Armed with this crucial result, we now finish our proof of Theorem \eqref{main prop}.
\begin{proof}[Proof of Theorem \eqref{main prop}]
It remains to show that the monomorphism $[\xan/\widehat{\mathbf{G}}_m]\hookrightarrow \X_1$ is an isomorphism. In view of Lemma \eqref{reduced is enough}, it suffices to show that the induced map between underlying reduced substacks
\begin{displaymath}
\left[\xan/\widehat{\mathbf{G}}_m\right]_{\mathrm{red}}\hookrightarrow (\X_1)_{\mathrm{red}}
\end{displaymath}
is an isomorphism. As our stacks will all live over $\Sp \mathbf{F}$ in the rest of this proof, we will drop the subscript $\mathbf{F}$ for ease of notation. For each ``Serre weight'' $\delta:I_K^{\mathrm{ab}}\to \mathbf{F}^\times$ (recall that $\mathbf{F}$ is assumed to contain $k$), we abusively denote also by $\delta$ a fixed choice of an extension of it to $G_K$. By twisting $\delta: \Sp \mathbf{F}\to \xan$ by unramified characters, we obtain a map $\mathbf{G}_{m}\to \xan$. The induced map $\coprod_{\delta} \mathbf{G}_{m}\to (\xan)_{\mathrm{red}}$ is then easily seen to be an isomorphism. In particular, we have an isomorphism $\coprod_{\delta} [\mathbf{G}_{m}/\mathbf{G}_m]\xrightarrow{\sim} [\xan/\widehat{\mathbf{G}}_m]_{\mathrm{red}}$. Thus, it suffices to show that the map 
\begin{displaymath}
\coprod_{\delta} \left[\mathbf{G}_{m}/\mathbf{G}_m\right]\hookrightarrow (\X_1)_{\mathrm{red}}
\end{displaymath}
is an isomorphism. Of course, by construction the component $[\mathbf{G}_m/\mathbf{G}_m]\to (\X_1)_{\mathrm{red}}$ indexed by $\delta$ is just obtained by twisting the residual gerbe $[\Sp \mathbf{F}/\mathbf{G}_m]\hookrightarrow (\X_1)_{\mathrm{red}}$ associated to $\delta$ by unramified characters. In particular, for $\delta=1$, we recover the map $[\mathbf{G}_m/\mathbf{G}_m]\hookrightarrow (\X_1)_{\mathrm{red}}$ induced by the universal unramified character $\ur{x}$.  

By Lemma \eqref{closed immersion crucial}, this last map is a closed immersion. After twisting by $\delta$, we see that the same is true of the map $[\mathbf{G}_m/\mathbf{G}_m]\hookrightarrow (\X_1)_{\mathrm{red}}$ indexed by $\delta$. As any character $\delta: G_K\to \overline{\mathbf{F}}_p^\times$ is an unramified twist of exactly one of the $\delta$, it is now straightforward (see Lemma \eqref{reduced substack} below) to deduce that the map $\coprod_{\delta}[\mathbf{G}_m/\mathbf{G}_m]\hookrightarrow (\X_1)_{\mathrm{red}}$ is indeed an isomorphism, as desired. 
\end{proof}
\begin{lem}\label{reduced substack}
Let $\Z$ be a reduced algebraic stack locally of finite type over a field $k$. Let $\Z_1,\ldots,\Z_n$ be a family of closed algebraic substacks of $\Z$ with the property that every $\bar{k}$-point of $\Z$ factors through exactly one of the $\Z_i$. Then the natural map $\coprod_i \Z_i\to \Z$ is an isomorphism.
\end{lem}
\begin{proof}
As usual, we denote by $|\Z|$ the underlying topological space of $\Z$, and similarly for $|\Z_i|$. We first show that $|\Z|=\coprod_i|\Z_i|$ set-theoretically. Let $\Z'$ be the scheme-theoretic image of the map $\coprod_i \Z_i\to \Z$. Then $\Z'$ is a closed algebraic substack of $\Z$ with $\Z'(\bar{k})=\Z(\bar{k})$. Since $\Z$ is reduced, this forces $\Z'=\Z$ (this can be checked after passing to a smooth cover of $\Z$ by a reduced scheme, where the result is standard). In particular, we have $|\Z|=|\Z'|$. As $|\Z'|$ is the closure of $\bigcup_i |\Z_i|$ in $|\Z|$ (cf. \cite[\href{https://stacks.math.columbia.edu/tag/0CML}{Tag 0CML}]{Sta21}), and each $|\Z_i|$ is a closed subset of $|\Z|$, we see that $|\Z|=\bigcup_i |\Z_i|$. Now for each $i\ne j$, $\Z_i\times_{\Z}\Z_j$ is an algebraic stack locally of finite type over $k$ with $(\Z_i\times_{\Z}\Z_j)(\bar{k})=\emptyset$ by our assumption. This forces $|\Z_i|\cap |\Z_j|=|\Z_i\times_{\Z}\Z_j|=\emptyset$. Thus we have a disjoint decomposition $|\Z|=\coprod_i |\Z_i|$, and hence each $|\Z_i|$ is also an open subset of $|\Z|$. Let $\U_i$ be the unique \textit{open} substack of $\Z$ with underlying space $|\U_i|=|\Z_i|$ (cf. \cite[\href{https://stacks.math.columbia.edu/tag/06FJ}{Tag 06FJ}]{Sta21}). In particular, we have a decomposition $\Z=\coprod_i \U_i$ into open substacks. Now for each $i$, the map $\Z_i\hookrightarrow \Z$ necessarily factors through a closed immersion $\Z_i\hookrightarrow \U_i$. Since $\U_i$ is reduced (being an open substack of $\Z$) and $|\Z_i|\xrightarrow{\sim}|\U_i|$ by construction, this closed immersion is necessarily an isomorphism.   
\end{proof}
\begin{remark}\label{independent of frob}
Assume $f: \xan\to \X_1$ is a morphism of stacks over $\Spf \O$ with the property that $f(\delta)\cong \aka(\delta)$ for all characters $\delta$ valued in a finite Artinian $\O$-algebra. We claim that in fact $f(\delta)\cong \aka(\delta)$ everywhere, or equivalently, that the map $g: \xan\to \X_1, \delta\mapsto f(\delta)\aka(\delta)^{-1}$ satisfies $g(\delta)\cong \aka$ for any $\delta$. Indeed, by \cite[Lem. 7.1.14]{EG22}, our assumption on $f$ implies that $g$ factors through the closed immersion $[\Spf \O/\widehat{\mathbf{G}}_m]\hookrightarrow \X_1$ (induced by the map $\Spf \O\hookrightarrow \xan$ classifying trivial characters). It therefore suffices to show that any map $\xan\to [\Spf \O/\widehat{\mathbf{G}}_m]$ is ``trivial'', i.e. that any line bundle on the Noetherian affine formal scheme $\xan$ is trivial. For this, it suffices to check the same result for the underlying reduced subscheme $(\xan)_{\mathrm{red}}$. But we have seen that $(\xan)_{\mathrm{red}}$ is just a disjoint union of finitely many copies of $\mathbf{G}_{m,\mathbf{F}}=\Sp \mathbf{F}[x,x^{-1}]$, and hence (as $\mathbf{F}[x,x^{-1}]$ is a PID), we have the claimed result. Thus we see that there is a unique functorial way to extend the construction $\delta\mapsto \aka(\delta)$ appearing in Dee's equivalence \eqref{Dee} from Artinian coefficients to all $\varpi$-adically complete $\O$-algebras $A$.
\end{remark}
\subsubsection{Proof of Lemma \texorpdfstring{\eqref{closed immersion crucial}}{3.6}}\label{proof of crucial}
\begin{proof}[First proof of Lemma \texorpdfstring{\eqref{closed immersion crucial}}{3.6}]
As in Remark \eqref{monomorphism}, the map $\ur{x}: \widehat{\mathbf{G}}_m\to \X_1$ induces a monomorphism 
\begin{displaymath}
    \ur{x}: [\widehat{\mathbf{G}}_m/\widehat{\mathbf{G}}_m]\hookrightarrow \X_1.
\end{displaymath}
(More formally, this map is given by composing the monomorphism $[\xan/\widehat{\mathbf{G}}_m]\hookrightarrow \X_1$ with the closed immersion $[\widehat{\mathbf{G}}_m/\widehat{\mathbf{G}}_m]\hookrightarrow [\xan/\widehat{\mathbf{G}}_m]$ induced by the closed immersion $\widehat{\mathbf{G}}_m\hookrightarrow \xan$ classifying unramified characters.) We want to show that this is in fact a closed immersion. As proper monomorphisms are closed immersions, it suffices to show that $\ur{x}$ is (representable by algebraic spaces and) proper. Note also that it suffices to work over $\O/\varpi^a$ for some $a\geq 1$ (for our proof of Theorem \eqref{main prop} it suffices to take $a=1$, but this will not simplify the argument). 

Let $\C_{1,0}$ be the stack of rank 1 projective Breuil--Kisin modules over $\fkS_A$ of \textit{height at most $0$} (in the terminology of \cite{EG22}). Concretely, objects of $\C_{1,0}$ are rank 1 projective $\fkS_A$-modules $\fM$ equipped with an \textit{isomorphism} $\varphi^*\fM\simeq \fM$; here $\fkS_A$ denotes the ring $(W(k)\otimes_{\mathbf{Z}_p}A)[[t]]$, equipped with the $A$-linear Frobenius $\varphi$ taking $t\mapsto t^p$ (and restricting to the natural Frobenius on $W(k)$). Let $\R_1$ be the the corresponding stack of rank 1 projective \' etale $\varphi$-modules over $\O_{\E,A}$, where $\O_{\E,A}:=\fkS_A[1/t]$. Our strategy is to relate $\ur{x}$ to the natural map $\C_{1,0}\to \R_1, \fM\mapsto \fM[1/t]$, which is known to be proper. Intuitively, as the source of $\ur{x}$ classifies unramified characters, and as crystalline representations (e.g. unramified characters) are of finite height, it is natural to guess that the composition $[\mathbf{G}_m/\mathbf{G}_m]\xrightarrow{\ur{x}}\X_1\to \R_1$ factors through the map $\C_{1,0}\to \R_1$. Here $\X_1\to \R_1$ is the natural map given by ``restriction to $G_{K_{\infty}}\subseteq G_K$'', where $K_\infty:=K(\pi^{1/p^{\infty}})$ for a compatible system $\pi^{1/p^\infty}$ of $p$-power roots of some fixed uniformizer $\pi$ of $K$ (see \cite[\textsection 3.7]{EG22}). We claim that this is indeed the case. More precisely, we will show that there is an isomorphism $[\mathbf{G}_m/\mathbf{G}_m]\simeq \C_{1,0}$ making the diagram 
\begin{equation}\label{commute}
\begin{tikzcd}
{[\mathbf{G}_m/\mathbf{G}_m]}\ar[d,"\ur{x}",swap]\ar[r,"\simeq"] & \C_{1,0}\ar[d]\\
\X_1\ar[r] & \R_1
\end{tikzcd}
\end{equation}
commute. Indeed, as finite projective modules over $\fkS_A$ are Zariski locally free on $\Sp(A)$ (see e.g. \cite[Prop. 5.1.9]{EG19}), we see that 
\begin{displaymath}
\C_{1,0}\simeq [LG^+/_{\varphi}]
\end{displaymath} 
where $LG^+$ denotes the functor $A\mapsto\fkS_A^\times$ and the quotient $/_{\varphi}$ is via the action of $LG^+(A)$ on itself by $\varphi$-twisted conjugation: $g\cdot M:=gM\varphi(g)^{-1}$. We have $\fkS_A=(W(k)\otimes_{\mathbf{Z}_p}A)[[t]]\simeq \prod_{0\leq j\leq f-1}A[[t]], x\otimes a\mapsto (\varphi^j(x)a)_j$. Under this identification, the action of $\varphi$ on $\fkS_A$ is given by $\varphi: (h_j)_j\mapsto (\varphi(h_{j+1}))_j$; here we also denote by $\varphi$ the $A$-linear action on $A[[t]]$ taking $t\mapsto t^p$. It is then easy to see that the map $\prod_j A[[t]]^\times \to A[[t]]^\times, (h_j)_j\mapsto \prod_j \varphi^j(h_j)$ induces an equivalence of groupoids
\begin{align*}
LG^+(A)/_{\varphi} &\simeq A[[t]]^\times/_{\varphi^f}.
\end{align*}
Now since $A[[t]]^\times=A^\times \times (1+t A[[t]])$ and elements in $1+tA[[t]]$ are ``killed'' in the quotient $/_{\varphi^f}$ (given any $h\in 1+tA[[t]]$, the series $g:=\prod_{n\geq 0} \varphi^{fn}(h)\in 1+tA[[t]]$ is well-defined and satisfies $h=g/\varphi^f(g)$), we obtain $A[[t]]^\times/_{\varphi^f}\simeq A^\times/A^\times$ with $A^\times$ acting trivially on itself, whence $\C_{1,0}\simeq [\mathbf{G}_m/\mathbf{G}_m]$. Unwinding definitions, one checks that a quasi-inverse is given by $a\in A^\times \mapsto D_{k,a}\otimes_{W(k)\otimes_{\mathbf{Z}_p}A}\fkS_A$, where $D_{k,a}$ is the rank one $\varphi$-module from Lemma \eqref{universal unramifed char lem}. This description also implies that the diagram \eqref{commute} commutes (again by unwinding definitions), as claimed.

In conclusion, we deduce that the composition $[\mathbf{G}_m/\mathbf{G}_m]\xrightarrow{\mathrm{ur}_x}\X_1\to \R_1$ is proper by \cite[Thm. 5.4.11 (1)]{EG19}\footnote{In this context (of usual Breuil--Kisin modules), a similar properness is first proved in \cite[Cor. 2.6]{PR09}. Note however that the definition of \' etale $\varphi$-modules in \cite{EG19} (and \cite{EG22}) is slightly less restrictive than that in \cite{PR09}: namely, in \textit{loc. cit.}, the authors demand furthermore that such a module $M$ is free \textit{fpqc} locally on $\Sp(A)$. In the rank 1 case, it follows \textit{a posteriori} from the explicit description in Corollary \eqref{description for etale phi-modules} below that this local freeness is a consequence of $M$ being projective (and even holds Zariski locally), and hence the two definitions yield the same stack; in general we do not know whether or not this is true.}. As the diagonal of the second map is a closed immersion by \cite[Prop. 3.7.4]{EG22}, the usual graph argument shows that the first map is proper, as wanted.
\end{proof}
The following proof was actually our first approach; it is somewhat more complicated than the previous proof in that it makes use of the existence of a certain crystalline substack of $\X_1$.
\begin{proof}[Second proof of Lemma \texorpdfstring{\eqref{closed immersion crucial}}{3.6}]
We want to show that the monomorphism $\ur{x}: [\widehat{\mathbf{G}}_m/\widehat{\mathbf{G}}_m]\hookrightarrow \X_1$ is a closed immersion. First recall that by \cite[Thm. 4.8.12]{EG22}, $\X_1$ admits a closed $\O$-flat $p$-adic formal algebraic substack $\X_1^{\mathrm{ur}}$, which is uniquely characterized by the property that, for any finite flat $\O$-algebra $\Lambda$, $\X_1^{\mathrm{ur}}(\Lambda)$ is the subgroupoid of $\X_1(\Lambda)$ consisting of characters $G_K\to \Lambda^\times$ which are (after inverting $p$) crystalline of Hodge--Tate weights $0$, or equivalently, unramified characters $G_K\to \Lambda^\times$. (Note that we are free to enlarge the field of coefficients $E$; in particular we may assume that it contains the Galois closure of $K$ so that the running assumption of \cite[Thm. 4.8.12]{EG22} is satisfied.) We will show that the map $\ur{x}$ factors through an isomorphism $[\widehat{\mathbf{G}}_m/\widehat{\mathbf{G}}_m]\simeq \X_1^{\mathrm{ur}}$, proving the lemma. Before continuing, we note that here it will be crucial to work directly over $\O$ (as opposed to the previous proof where we work modulo $\varpi^a$ for some $a\geq 1$).

We begin by showing that $\ur{x}$ factors through the closed substack $\X_1^{\mathrm{ur}}$, or equivalently, that the closed immersion 
\begin{displaymath}
    \X_1^{\mathrm{ur}}\times_{\X_1}[\widehat{\mathbf{G}}_m/\widehat{\mathbf{G}}_m]\hookrightarrow [\widehat{\mathbf{G}}_m/\widehat{\mathbf{G}}_m]
\end{displaymath}
is an isomorphism. As the target is a $p$-adic formal algebraic stack of finite type and flat over $\Spf \O$ (since it admits a smooth cover by the $p$-adic formal algebraic space $\widehat{\mathbf{G}}_m$, which is of finite type and flat over $\Spf \O$), it follows from \cite[Lem. 7.2.6 (3)]{LHLMlocal} that it suffices to show that for any morphism $\Spf \Lambda\to [\widehat{\mathbf{G}}_m/\widehat{\mathbf{G}}_m]$ whose source is a finite flat $\O$-algebra (endowed with the $p$-adic topology), the composite $\Spf \Lambda\to [\widehat{\mathbf{G}}_m/\widehat{\mathbf{G}}_m]\hookrightarrow \X_1$ factors through $\X_1^{\mathrm{ur}}$. This is clear as $[\widehat{\mathbf{G}}_m/\widehat{\mathbf{G}}_m](\Lambda)$ is also equivalent (via the monomorphism $[\widehat{\mathbf{G}}_m/\widehat{\mathbf{G}}_m]\hookrightarrow \X_1$) to the subgroupoid of $\X_1(\Lambda)$ consisting of unramified characters $G_K\to \Lambda^\times$. 

It remains to show that the induced monomorphism 
\begin{displaymath}
    [\widehat{\mathbf{G}}_m/\widehat{\mathbf{G}}_m]\hookrightarrow \X_1^{\mathrm{ur}}.
\end{displaymath}
is in fact an isomorphism. As the source and target are both $p$-adic formal algebraic stacks which are of finite type and flat over $\Spf \O$, and moreover $\X_1^{\mathrm{ur}}$ is analytically unramified by Lemma \eqref{x1 ur unramified} below, it suffices, by \cite[Lem. 7.2.6 (1)]{LHLMlocal}, to check that the above map induces an isomorphism on any finite flat $\O$-algebra $\Lambda$. This follows again from the fact that both sides admit the same moduli description (namely, as unramified characters) on these points. 
\end{proof}
The following lemma is presumably standard, but for lack of a reference, we include a proof here. 
\begin{lem}\label{unramified reduced}
Let $\X$ be a $p$-adic formal algebraic stack locally of finite type over $\Spf \O$. If $\X$ admits reduced Noetherian versal rings at all finite type points, then $\X$ is (residually Jacobson and) analytically unramified in the sense of \cite[Rmk. 8.23]{Eme}. 
\end{lem}
\begin{proof}
Let $\X'$ be the associated reduced formal algebraic substack of $\X$, as defined in \cite[Ex. 9.10]{Eme}. We claim that $\X'$ is analytically unramified. Choose a smooth (in particular, representable by algebraic spaces) surjection $\coprod_i \Spf B_i\to \X$ where each $B_i$ is a $p$-adically complete $\O$-algebra. By construction of $\X'$, we know that the base change $\X'\times_{\X}\Spf B_i$ is identified with $\Spf (B_i)_{\mathrm{red}}$ (see \cite[Ex. 9.10]{Eme}), and furthermore that each $\Spf (B_i)_{\mathrm{red}}$ is analytically unramified (e.g. by \cite[Cor. 8.25]{Eme}, applied to the finite type adic map $\Spf (B_i)_{\mathrm{red}}\to \Spf \O$). As $\X'$ receives a smooth surjection from the disjoint union $\coprod_i\Spf (B_i)_{\mathrm{red}}$ of analytically unramified affine formal algebraic spaces, it is analytically unramified by definition. 

We now show that the closed immersion $\X'\hookrightarrow \X$ is in fact an isomorphism (this will imply that $\X$ is analytically unramified, as desired). As this can be checked at the level of Artinian points (e.g. by \cite[Lem. 7.1.14]{EG22}), it in turn suffices to work with versal rings. More precisely, let $x\in \X(A)$ be a point valued in a finite Artinian local $\O$-algebra $A$. By assumption, $\X$ admits a reduced Noetherian versal ring $\Spf B$ at the finite type point induced by $x$. By versality, the map $x: \Sp A\to \X$ factors through $\Spf B\to \X$, so it suffices to show that the latter map factors through $\X'$. Choose a smooth cover $\coprod_i \Spf B_i\to \X$ as before. It suffices to show that each base change $\Spf B\times_{\X}\Spf B_i\to \Spf B_i$ factors through $\Spf (B_i)_{\mathrm{red}}$, which in turn will follow once we show that given any smooth morphism $\Spf C\to \Spf B\times_{\X}\Spf B_i$, the composite $\Spf C\to \Spf B\times_{\X}\Spf B_i\to \Spf B_i$ factors through $\Spf (B_i)_{\mathrm{red}}$. As any ring map from $B_i$ to a reduced ring necessarily factors through $(B_i)_{\mathrm{red}}$, it suffices to show that $C$ is reduced. As $B$ is complete local Noetherian and reduced, it is analytically unramified by definition. The upshot is that we have a smooth morphism $\Spf C\to \Spf B$ whose target is analytically unramified (and residually Jacobson, as $(\Spf B)_{\mathrm{red}}=\Sp B/\m_B$ is just a point). It then follows from \cite[Lem. 8.20]{Eme} that the source $\Spf C$ is also analytically unramified, and in particular that $C$ is reduced, as required. 
\end{proof}
\begin{lem}\label{x1 ur unramified}
The stack $\X_1^{\mathrm{ur}}$ is analytically unramified.
\end{lem}
\begin{proof}
This follows from Lemma \eqref{unramified reduced}, \cite[Prop. 4.8.10]{EG22} and \cite[Thm. 3.3.8]{Kisin08}.
\end{proof}
\subsection{The case of \texorpdfstring{\' etale $\varphi$-modules}{etale phi-modules}}\label{etale phi description}
We end this note by briefly explaining how our method can also be used to give explicit descriptions of the stacks of rank one \' etale $\varphi$-modules (i.e. in the absence of a $\Gamma$-action), generalizing \cite[Prop. 7.2.11]{EG22} to a large class of coefficient rings. We will put ourselves in the context of Situation 2.2.15 in \cite{EG22}. Namely, fix a finite field $k\simeq \mathbf{F}_{p^f}$ and write $\mathbf{A}^+:=W(k)[[t]]$. Let $\mathbf{A}$ be the $p$-adic completion of $\mathbf{A}^+[1/t]$. Let $\varphi: \mathbf{A}\to \mathbf{A}$ be a ring map lifting the Frobenius modulo $p$ (we do not assume that $\varphi$ preserves the subring $\mathbf{A}^+$). We note that $\varphi$ necessarily induces the natural Frobenius on $W(k)$. If $A$ is a $p$-adically complete $\mathbf{Z}_p$-algebra, we write $\mathbf{A}_A^+:=(W(k)\otimes_{\mathbf{Z}_p}A)[[t]]$ and let $\mathbf{A}_A$ be the $p$-adic completion of $\mathbf{A}_A^+[1/t]$, equipped with the $A$-linear extension of $\varphi$. Let $\R_1$ denote the stack of rank 1 projective \' etale $\varphi$-modules over $\mathbf{A}_A$, as defined in \cite[\textsection 3.1]{EG22} (we continue to assume that the coefficient ring $\O$ is large enough so that $k\hookrightarrow\mathbf{F}$). By $W_{k((t))}$ (resp. $G_{k((t))}$), we will mean the Weil group (resp. the Galois group) of the local field $k((t))$.

\begin{prop}
    There is a natural isomorphism
    \begin{displaymath}
\left[\xan_{k((t))}/\widehat{\mathbf{G}}_m\right]\xrightarrow{\sim} \R_1;
    \end{displaymath}
here $\xan_{k((t))}$ denotes the functor on $p$-adically complete $\O$-algebras taking $A$ to the set of continuous characters $\delta: W_{k((t))}\to A^\times$, and in the formation of the quotient stack, the $\widehat{\mathbf{G}}_m$-action is taken to be
trivial.
\end{prop}
\begin{proof}
Since the proof is very similar to that of Theorem \eqref{main prop}, we will content ourselves with indicating the main steps.
\begin{itemize}
    \item Recall firstly that Dee's equivalence \eqref{Dee} also admits a variant for \' etale $\varphi$-modules: if $A$ is a finite Artinian $\O$-algebra, then there is a natural equivalence from the category of finite projective \' etale $\varphi$-modules over $\mathbf{A}_A$, to the category of finite free $A$-modules with a continuous action of $G_{k((t))}$ (see \cite[Thm. 2.1.27]{Dee01}).
    \item Construct a map $\xan_{k((t))}\to \R_1, \delta\mapsto \mathbf{A}_A(\delta)$ extending Dee's equivalence from Artinian coefficients to all $p$-adically complete $\O$-algebras $A$; the main point is again the construction of the universal unramified character, which can be done similarly as in Definition \eqref{universal unramified} (namely, given $a\in A^\times$, we define $\mathbf{A}_A(\ur{a}):=D_{k,a}\otimes_{W(k)\otimes_{\mathbf{Z}_p}A}\mathbf{A}_A$, where $D_{k,a}$ is the $\varphi$-module from Lemma \eqref{universal unramifed char lem}). 
    \item Show that the automorphisms of a rank one \' etale $\varphi$-modules are simply the scalars. This amounts to showing that $(\mathbf{A}_A)^{\varphi=1}=A$ for any $\O/\varpi^a$-algebra $A$. The case $a=1$ is clear as then $\varphi(t)=t^p$. The general case then follows easily from this by d\' evissage. 
    
    Combining with the previous item, one obtains a natural monomorphism 
    \begin{equation}\label{monomorphism etale phi}
        [\xan_{k((t))}/\widehat{\mathbf{G}}_m]\hookrightarrow \R_1,
    \end{equation}
    which we want to show is an isomorphism.
    
    \item Reduce to proving essential surjectivity of \eqref{monomorphism etale phi} for reduced test rings. The proof is similar to that of Lemma \eqref{reduced is enough}, except that we have to be slightly more careful as we do not know a priori if $\R_1$ is limit preserving in general (recall that we do not assume that $\mathbf{A}^+$ is $\varphi$-stable). To overcome this, the idea is to reduce first to the case $A$ is an $\mathbf{F}$-algebra by mimicking the argument there for the ideal $\varpi A$. Note that the set of isomorphism classes of objects in $\R_1(A)$ lifting the trivial object in $\R_1(A/I)$ ($I$ being a square-zero ideal) is now given by $\mathbf{A}_I/(\varphi-1)$, or more precisely, by $H^1$ of the complex $[\mathbf{A}_I\xrightarrow{\varphi-1}\mathbf{A}_I]$ (concretely, any such object admits an $\mathbf{A}_A$-basis $v$ so that $\varphi(v)=fv$ for some $f\in \ker(\mathbf{A}_A^\times\to \mathbf{A}_{A/I}^\times)=1+\mathbf{A}_I$; the corresponding cohomology class is then given by $[f-1]$). Again, here we cannot immediately reduce to the case where $I$ is a finite $\O$-module as in the proof of Lemma \eqref{reduced is enough}. Instead, we will show directly that if $\varpi I=0$ (which is the only case that we need), then $\mathbf{A}_I/(\varphi-1)=\varinjlim \mathbf{A}_{I_i}/(\varphi-1)$ where $\{I_i\}$ is the system of the finite sub-$\mathbf{F}$-modules of $I$. As $k\otimes_{\mathbf{F}_p}\mathbf{F}\simeq \prod\mathbf{F}$, we have $\mathbf{A}_I/(\varphi-1)\simeq I((t))/(\varphi^f-1)$ (here $\varphi(t)=t^p$ as $\varpi I=0$). As elements in $tI[[t]]$ are killed in the quotient (given $h\in tI[[t]]$, the series $g:=\sum_{n\geq 0} \varphi^{fn}(h)\in tI[[t]]$ is well-defined and satisfies $h=(1-\varphi^f)(g)$), it suffices to prove the analogous claim for $I[1/t]$, which is obvious (as we are working with polynomials, as opposed to (infinite) series). Finally, over $\mathbf{F}$, the subring $\mathbf{A}^+$ is $\varphi$-stable (as $\varphi(t)=t^p$), and we can invoke \cite[Thm. 5.4.11 (3)]{EG19} to deduce that $(\R_1)_{\mathbf{F}}$ is limit preserving, and hence reduce to the case $A$ is a finite type $\mathbf{F}$-algebra. We now run the same argument, but for the (nilpotent) ideal $A^{00}$.
    
    \item By the previous item, it suffices to show that the map \eqref{monomorphism etale phi} induces an isomorphism on underlying reduced substacks. Again as there are only finitely many mod $p$ characters of $G_{k((t))}$ up to unramified twist, we are reduced to showing that the map 
    \begin{equation}\label{closed immersion etale phi}
        [\mathbf{G}_m/\mathbf{G}_m]_{\mathbf{F}}\hookrightarrow (\R_1)_{\mathbf{F}}
    \end{equation}
    induced by the universal unramified character is a closed immersion (see Lemma \eqref{reduced substack}). Again, over $\mathbf{F}$, the subring $\mathbf{A}^+$ is $\varphi$-stable (this is automatic in our first proof of Lemma \eqref{closed immersion crucial}), and we can invoke \cite[Thm. 5.4.11 (1)]{EG19} to see that the natural map $\C_{1,0}\to \R_1$ is proper, where the source denotes the stack (over $\mathbf{F}$) of rank 1 projective $\varphi$-modules over $\mathbf{A}_A^+$ of height at most $0$.  The rest of the proof is now as before: namely, we show that the map \eqref{closed immersion etale phi} factors as $[\mathbf{G}_m/\mathbf{G}_m]\xrightarrow{\simeq} \C_{1,0}\to \R_1$ with the first map being given by $a\in A^\times\mapsto D_{k,a}\otimes_{W(k)\otimes_{\mathbf{Z}_p}A}\mathbf{A}_A^+$. 
\end{itemize}
\end{proof}
Specializing to the case $\mathbf{A}=\mathbf{A}_K$, and combing with the Fontaine--Wintenberger isomorphism $G_{K_{\mathrm{cyc}}}\simeq G_{k_{\infty}((t))}$ (which restricts to an isomorphism of Weil subgroups), we recover the following result.
\begin{cor}[{\cite[Prop. 7.2.11]{EG22}}]\label{description for etale phi-modules}
There is an isomorphism 
\begin{displaymath}
    \left[\left(\Spf \O[[I_{K_{\mathrm{cyc}}}^{\mathrm{ab}}]]\times \widehat{\mathbf{G}}_m\right)/\widehat{\mathbf{G}}_m\right]\xrightarrow{\sim} \R_{K,1}
\end{displaymath}
(again, in the formation of the quotient stack, the $\widehat{\mathbf{G}}_m$-action is taken to be
trivial).
\end{cor}
\medskip
\printbibliography[title ={References}]
\addcontentsline{toc}{section}{References}
\textsc{\small LAGA, Universit\' e Paris 13, 99 avenue Jean-Baptiste Cl\' ement, 93430 Villetaneuse, France}\\
\indent\textit{Email address}: \href{mailto:dat.pham@math.univ-paris13.fr}{\texttt{dat.pham@math.univ-paris13.fr}}
\end{document}